\documentclass[12pt,a4paper]{article}
\usepackage{amsfonts}
\usepackage{amsmath}
\usepackage{amsbsy}
\usepackage{amsxtra}
\usepackage{latexsym}
\usepackage{amssymb}
\usepackage{amsthm}
\usepackage{url}
\usepackage{cite}
\usepackage{pstricks,pst-node}
\usepackage{enumerate}
\makeatletter
\g@addto@macro\thesection
\makeatother

\newtheorem{theorem}{Theorem}

\newtheorem{lemma}{Lemma}

\newtheorem{proposition}{Proposition}
\newtheorem{remark}{Remark}
\newtheorem{definition}{Definition}
\newtheorem{example}{Example}

\DeclareMathOperator{\cone}{cone}

\DeclareMathOperator{\co}{co}

\DeclareMathOperator{\argmin}{argmin}

\DeclareMathOperator{\spa}{span}

\newcommand{\lng}{\langle}
\newcommand{\rng}{\rangle}
\newcommand{\R}{\mathbb R}

\newcommand{\mc}{\mathcal}

\newcommand{\Hyp}{\mathcal H}

\newcommand{\Conv}{\mathcal C}

\newcommand{\Kup}{\mathcal K}

\newcommand{\Aa}{\mathcal A}

\begin{document}

\sffamily

\date{}

\title{The metric geometry of subspaces and convex cones of the Banach space revisited}
\author{A. B. N\'emeth\\Faculty of Mathematics and Computer Science\\Babes Bolyai University, Str. Kog\u alniceanu nr. 1-3\\RO-400084
Cluj-Napoca, Romania\\email: nemeth{\huge\_}sandor@yahoo.com}

\maketitle

\begin{abstract}
Using some earlier and recent results, it is proved, among other things, that in a uniformly convex and uniformly smooth
 real Banach space $X$, the following three assertions are equivalent:

1.The metric projections onto the intersections of pairs of closed hyperplanes are linear;
2.The polars of the intersections of pairs of closed halfspaces are convex;
3.The space is an inner product space.

Hence the linearity of the metric projections onto subspsces of $X$ and the convexity
of the polars of closed convex cones of $X$ are equivalent, and booth are equivalent with 
the inner-product  generatedness of the norm of $X$. 

 \end{abstract}
\section{Introduction}

There is an increasing interest in the metric geometry of subspaces and convex cones in Banach spaces. 
Notable contributions include recent papers \cite{DM}, \cite{KL}, \cite{KL24},  and 
\cite{KhanKongLi2025}.

On a different note, a recent investigation into a very general notion of projection in Euclidean spaces \cite{NemethNemeth2025} raises several convex geometric problems that seem to have their natural place in the geometry of Banach spaces \cite{Nemeth2025}.

This note aims to revisit some earlier results in approximation theory in normed vector spaces \cite{Singer} and, building on \cite{Nemeth2025}, provide a unified approach to the linearity of the metric projection onto subspaces of a Banach space and the convexity of the polars of closed convex cones in this space.


\subsection{The metric projections on closed convex sets}

The nonempty  set  $\Conv \subseteq X$ is called \emph{convex} if
$$u,\,v\in \Conv\,\implies \,[u,v]\subseteq \Conv \textrm{, where }
[u,v]=\{tu+(1-t)v :\, t \in [0,1] \subseteq \R\}.$$

We will  also use the notation for $u,v \in X, u\not=v,\,\,[u,v =\{u+t(v-u),\,t\in \R_+\}.$

If $\Aa \subseteq X$ is a nonempty set,
$$\co \Aa$$
is he smallest  convex set containing $\Aa $ and is called
\emph{the convex hull of $\Aa$.}

Denote by $(X,\|.\|)$ an uniformly convex and uniformly smooth  Banach space over the reals with its norm dual $(X^*,\|.\|^*)$.
 Then for every nonempty closed convex set $\Conv \subset X$ the mapping  
$P_\Conv:X\to \Conv$ given by
	$$P_\Conv x= \argmin \{\|x-c \|: c\in \Conv\}$$
is well defined and single valued.	 $P_\Conv$ is called the \emph{metric projection} of $X$ onto $\Conv$.


\subsection{Hyperplanes and halfspaces}

The set $\Hyp\subseteq  X$ is said \emph{a hyperplane with the normal
$a,\,a\in X^*, \|a\|=1$} if 
$$\Hyp=\{x\in X:\,\lng a,x \rng=0\}.$$ A hyperplane defines two 
\emph{closed halfspaces} defined by
$$\Hyp_+=\{x\in X :\,\lng a,x \rng \geq 0\},\,\,
\Hyp_-=\{x\in X:\,\lng a,x \rng \leq 0\},$$
and two \emph{open halfspaces} defined by
$$\Hyp^+=\{x\in X:\,\lng a,x \rng > 0\},\,\,
\Hyp^-=\{x\in X:\,\lng a,x \rng < 0\}.$$

The translate of the hyperplane $\Hyp$ (halfspace $\Hyp_+,\,\,\Hyp^-$, etc.), i. e.
the sets $\Hyp +u \,\,(\Hyp_+ +u,\,\Hyp^- +u, $ etc.) are also called hyperplane (halfspace).
We also use the notations \textit{}$\Hyp_-(a)=\Hyp_-$ (similarly for the hyperplane
and the other halfspaces) to emphasise that the normal of $\Hyp$ is $a$.

\subsection{Cones and convex cones}

The nonempty set $\Kup \subseteq X$ is called a \emph{cone} if 
 $\lambda x\in \Kup$, for all $x\in \Kup$ and $\lambda \in \R_+$.
	
$\Kup$ is called \emph{convex cone} if it a cone satisfying
 $x+y\in \Kup$, for all $x,y\in \Kup$.

If $\Kup$ is a closed convex cone, then
$$ \Kup^\circ =\{x\in X:\,P_\Kup x=0\}$$
is called the \emph{polar} of the cone $\Kup$ . .


\subsection{The linearity of the metric projections on subspaces\\
and the convexity of the polars of  convex cones}

In the case of $X = H$, a Hilbert space, the metric projection onto every
closed vector subspace in $H$ is linear
and the polar of each convex cone in $H$  is also a closed convex cone. These properties of the metric
projection are fundamental concepts in a vast theory with extensive applications
 (see, for example, the excellent monograph by E. Zarantonello \cite{Zarantonello1971}).
 
The problem of the linearity of the metric projection onto closed subspaces of a Banach space is quite different. While the metric projection onto hyperplanes in a Banach space is linear, the linearity of the metric projection onto every closed linear subspace does not hold
in general. This problem has been a significant concern since the middle of the last century, as illustrated in the monograph by I. Singer \cite{Singer}.

Considering the example of the three dimensional $l_3$ space
  A. Domokos and M. M. Marsh \cite{DM} have got unidimensional subspace admitting nonlinear metric projection.

At a first glance it seems rather surprising the nonlinearity of metric prosection of some subspace of the
three dimensional $l_3$ space due to
the well properties of this space and since it is well known
that the metric projection nonto hyperplanes (subspaces of defect one) in the general
uniformly smooth and uniformly convex Banach space is linear. (See e.g. \cite{Singer}, \cite{DM} or \cite{FN}.)
(This explain  for instance that the counterexample of Domokos and Marsh
must be necessarily unidimensional subspace.)

The methods proving the linearity of the metrfic  projection
on hyperplanes do not work for subspaces of defect greather as one.
  
Recently, A. A. Khan, D. Kong, and J. Li  \cite{KhanKongLi2025} observed that for the same well-behaved three-dimensional Banach space
the polar of some convex cone is not convex.

As we will see next the coincidence of the above two negative results is not suprissing:
they express the same phenomenon. Namely that the linearity of the metric projection on subspaces is equivalent
with the convexity of the polars of convex cones.


\subsection{ The origin of our method}

Let $\R^n$ be the Euclidean $n$-space.

A far-reaching generalizationin   $P_\Kup^f$ of the Euclidean metric projection  of $x\in \R^n$  onto the 
nonempty, closed convex cone $\Kup \subset \R^n$ is 
the minimum point in $\Kup$ of the function $ f(x-.)$ where  $f$ is a non-negative, continuous, strictly quasiconvex
function with $f(0)=0$. It was shown by O. Ferreira and S. Z. N\'emeth in \cite{FN} and \cite {FN1} that a lot of the properties
and problems concerning
the Euclidean metric projection have their natural places in this very general context too.

In a similar vein A. B. N\'emeth and S. Z. N\'emeth in the recent paper  \cite{NemethNemeth2025}
have observed that in this general case $P_\Kup^f$ and $I-P_\Kup^f$ are also mutually polar retraction on cones. The range
of $I-P_\Kup^f$ is always a cone, but it is convex only with strong conditios on $f$.
The range of $I-P_\Kup^f$ is quit the polar $\Kup^{\circ,f}$ of the cone $\Kup$.
One of the principal achivements of \cite{NemethNemeth2025} is the characterization of the function $f$ for which the polar
$\Kup^{\circ,f}$ is convex.

Using a convex geometric terminology 
the problem was settled  by introducing the ad hoc notion of 
the coherent meridian. The developed method has a simple Banach space formulation
using the notion of the normalized duality mapping. This was done in \cite{Nemeth2025}.


\subsection{The normalized duality mapping}

The existence of a well-defined projection onto non-empty closed convex sets is an impportant property of
 reflexive Banach spaces. These are the spaces for which the bidual $(X^*)^*$ is linearly isomorphic to $X$.

The uniqueness of the image of the projection characterizes strictly convex Banach spaces
.

Thus, to ensure the desirable properties of the metric projection, the ambient space for our subsequent considerations will be 
{\bf the uniformly concex uniformly smooth reflexive Banach space $(X,\|.\|)$ with the duality pairing $\lng.\rng$.}

\begin{definition}\label{dual}
In this space is well defined and basic notion the \emph{normalized duality map}
$J:X\to X^*$ given by the relation
$$\lng Jx,x\rng=\|Jx\|_{X^*} \|x\|_X = \|x\|_X^2=\|Jx\|_{X^*}^2.$$
\end{definition}

In our ambient space $J$ is  homeomorphism with
the inverse $J^*=J^{-1}.$
(See e.g.  \cite{Drag},  \cite{Cior},  \cite{DM} , \cite{Dei}.)


\subsection{The main result}

The prezent note is in fact an appendix to  our notes \cite{NemethNemeth2025} and \cite{Nemeth2025}
showing among other things the following assertion:

\begin{theorem}\label{eloz}
In the uniformly convex and uniformly smooth reflexive  
Banach space $X$ of dimension $n\geq 3$  the following assertions are equivalent:
\begin{enumerate}[(i)]

\item 
The metric projection on every nonempty linear and closed subspace of $X$  is linear;
 
 \item 
 
 The metric projection onto the intersection of
 pairs of closed hyperplanes is linear;

 \item
  
$J^*:X^* \to  X$ maps bidimensional subspaces of $X^*$ in bidimensional subspaces of $X$;

\item

The polars of the intersections of pairs of closed halfspaces are convex;

\item  
Each nonempty closed convex cone in $X$ possesses a convex polar;

\item

$X$ is an inner product space.

\end{enumerate}

\end{theorem}

\begin{remark}

The condition $\dim X \geq 3$ in the theorem is essential. If $\dim X \leq 2$ the assertion (vi)
of the theorem does not hold.

\end{remark}


\section{Mutually polar retractions: definition \\  and examples}\label{mutu}

We suppose in this secrion that $X$ is a real vector space of dimension $\geq 3$.

The  mapping $Q:X\to X$ is called \emph{retraction} if  it
is \emph{idempotent}, that is, if $Q^2=Q$.
Let $0$ be the zero mapping and $I$ the identity mapping of $X$. They 
are retractions.

\begin{definition}\label{muture} 
The mappings $P,R:\,X\to X$ are called
\emph{mutually polar retractions} if 
\begin{enumerate}[(i)]
	\item $P$ and $R$ are retractions, 
	\item $P+R=I$,
\end{enumerate}
\end{definition}

 We will use the notations $PX=\mc P$ and $RX=\mc R$.

 We will say that $P$ and $R$ 
 are \emph{polars of each other}.
 Also, the sets $\mc P$ and $\mc R$ are said {\em mutually polar}.

\begin{example}\label{triv}
The mappings $I$ and $0$ are mutually polar retractions.
\end{example}

\begin{example}\label{ex0}
If $\R^n$ is the Euclidean $n$-space
endowed wih the orhonormal system $e^1,...,e^n$ and the reference system it defines, then putting 
$x=t_1e^1+\dots+t_ne^n$, consider the mappings
$$Qx=|t_1|e^1+\dots+|t_n|e^n$$
and
$$Px=t_1^+e^1+\dots+t_n^+e^n, $$
where $t^+$ denotes the positive part of the real number $t$.

$Q$ and $P$ are both retractions on the positive orthant $\R^n_+$. 
$I-P$ is retraction too, but $I-Q$ is not.
\end{example}

\begin{example}\label{eex2}

If $\mc K\subseteq \R^n$ is a closed convex cone and $\mc K^\circ$ is the polar of  $\mc K$,
then  the metric projections $P_\mc K$ and $P_{\mc K^\circ}$ are mutually polar retractions.
This is the consequence of the following decomposition theorem of Moreau (see
\cite{Moreau1962}, \cite{Zarantonello1971}).

\begin{theorem} [Moreau]\label{moreau}
	Let  $\mc K,\,\,\mc L\subset \R^n$ be  two mutually polar closed 
	convex cones. Then, the following statements are equivalent:
	\begin{enumerate}
		\item[(i)] $z=x+y,~x\in \mc K,\,\,~y\in \mc L$ and $\lng x,y\rng=0$,
		\item[(ii)] $x=P_\mc K(z)$ and $y=P_\mc L(z)$.
	\end{enumerate}
\end{theorem}

\end{example}
 
\begin{example}\label{riesz}

If $(X,\mc K)$ is a Riesz space with the positive cone $\mc K$ then $P$ and $R$ defined by
$Px=x^+$ and $Rx=-x^-$ are mutually polar retractions. Here $\mc P=\mc K$ and $\mc R = -\mc K$.

\end{example}

\begin{proposition}\label{kov}

If $P$ and $R$ are mutually polar retractions, then
\begin{enumerate}[(i)]
	\item $PR=RP=0.$
	\item  $\mc P \cap \mc R = \{0\}$ and $\mc P + \mc R= X$,
	\item $P0=R0 =0$;
	\item $Px-Rx= 0$ if and only if $x=0$.
\end{enumerate}
    \end{proposition}

\begin{proof}
	$\,$

	\begin{enumerate}[(i)]
		\item From Definition \ref{muture}, it follows that
			\(PR=(P+R)R-R^2=R-R^2=0.\) Hence, by swapping $P$ and
			$R$, we also have $RP=0$.
		\item Let any $z\in\mc P\cap\mc R$. Hence, there exists
			$x,y\in X$ such that $z=Px=Ry$, which, by using the 
			previous item, implies $z=Px=P^2x=PRy=0$. The second
			equality is trivial.
		\item Suppose that $P0=x$ Then
			$$x=P0=P^20=Px,\,\,\Rightarrow\,\,x=Px+Rx=x+Rx,\,\,\Rightarrow \,\, Rx=0,\,\,$$
			$$\Rightarrow Rx=R^2x = R0,\,\,\Rightarrow \,\, x=Px+Rx=P0+R0 = 0.$$
		\item The equality $P0-R0=0$ is from item (iii). Suppose that $Px-Rx=0$. 
			Hence, by using the first equality of item (ii), we
			obtain $Px=Rx\in\mc P\cap\mc R=\{0\}$. Hence, the second 
			equality of item (ii) yields $x=Px+Rx=0$. 
	\end{enumerate}
\end{proof}
The followig result of general character from the linear algepra can be  verified directly:

\begin{lemma}\label{lpt}

If $\mc M$ and $\mc N$ are vector subspaces of $X$ with the properties
\begin{enumerate}[(i)]

\item $\mc M+\mc N = X$ and

\item $\mc M \cap \mc N = \{0\} $

then
\item
for every $x\in X$ there exist a unique element $Mx \in \mc M$ and a unique element $Nx \in \mc N$ such that
$$ x=Mx+Nx$$ and
\item the mappings $M$ and $N$ defined this way are linear operators.

\end{enumerate}
\end{lemma}

If $Q$ is a retraceion wuth $\mc Q$ a cone , then  $Q$ will be called \emph{cone retract}.

\begin{proposition}\label{kovlin}

If $P$ and $R$ are mutually polar retractions, then
\begin{enumerate}[(i)]
	\item They are together cone retracts, and they are together linear;
	
	\item  $P$ and $R$ are linear if
	and only if $\mc P$ and $\mc R$  are vevtor spaces;

\end{enumerate}
    \end{proposition}

\begin{proof}
	$\,$

	\begin{enumerate}[(i)]
		\item From Definition $P=I-R$ and since $I$ is liear the assertions follow.
		
		\item If $P$ and $R$ are linear then their range $\mc P$ and $\mc R$ are vector spaces.
		
		Suppose that $\mc P$ and $\mc R$  are vector spaces. Then from the condition of  polarity of $P$ and $R$
		 in  item (ii) of Proposition \ref{kov} via Lemma \ref{lpt} it follows that  $P$ and $R$ are linear.
		  
	\end{enumerate}
\end{proof}


\section{Projection on the  nonempty closed convex cone}

\begin{theorem}\label{footet}
If $\Kup \subset X$ is a nonempty, closed convex cone in the uniformly convex
and uniformly smooth Banach space $X$ of dimension $\geq 3$, then  we have for the metric projection $P= P_\Kup$ that
\begin{equation}\label{foo1}
	P \,\, \textrm{and}\,\, R= I-P
\end{equation}
are mutually polar retractions and hence also $\Kup^\circ = RX$.
\end{theorem}

\begin{proof}

The sum of $P$ and $R$ is obviously the identity mapping.

We have only to show that $R$ is idempotent.

	Since $\Kup$ is closed under the addition, $Px+v\in \Kup$ for
	any $v\in \Kup$ and since $$Px=\argmin\{ \|x-u\| :u\in \Kup \},$$
we have $$ \|x-Px\| \leq \|x-Px-v\|$$
and since the projection is unique, it  follows that
$$\argmin\{\|x-Px-v\|: v\in \Kup \}=0.$$

Hence,
$$ P(x-Px) =\argmin\{ \|(x-Px)-u\|: u\in \Kup \} = \argmin\{ \| x-Px-u\|:u\in \Kup \}=0,\,\,\forall \, x$$
according to the above consideration. Thus
\begin{equation}\label{uf}
PR=P(I-P)=0.
\end{equation}

From (\ref{uf}) we have

$R^2= (I-P) (I-P) = I-P-P(I-P) = I-P= R.$

 \end{proof}
 
 
\section{The polar of the hyperplane and of the hypercone}

Let $(X,\|.\|)$ uniformly convex and uniformly smooth  reflexive Banach space of dimension $\geq 3$, 
   $(X^*,\|.\|^*)$ its norm dual and with $\lng . \rng$
the associated duality mapping . In this space every nonempty closed convex set is the intersection of closed halfspaces 
(\cite{Singer}, \cite{Cior}, \cite{Drag}).

 Denote
$S$ and respectively $S^*$ the \emph{unit spheres} in these spaces, that is:
$$ S=\{x\in X:\, \|x\|=1\} \,\, \textrm{and}\,\, S^*=\{x^*\in X^*:\,\|x^*\|^*=1\}.$$

The cone $\Kup \subset X$ is called a \emph{hypercone} if $\Kup \not= X$ and there
is no convex cone except $X$ containing it properly. Since the cloasure of a convex cone
is always a closed convex cone, it follows that the hypercone must be closed. Since $\Kup$ is the intersection of
closed halfspaces, it follows that there exists $a\in S^*$ such that
$$\Kup=\Hyp_-(a),\,\, a\in S^*,$$
that is a hypercone is in fact a closed halfspace whose supporting hyperplane contains $0$.

\begin{proposition}\label{felt}
If $$\Hyp(a)=\{x\in X:\,\lng a,x \rng=0\}$$    
is a hyperplane in $X$, it can be written using $J$ in the form
$$\Hyp(a)=\{x\in X:\,\lng J(y),x \rng=0\}$$
for some $y\in S$.
Thus any hypercone  $\Hyp_-(a)$ can be written in the form $\Hyp_-(Jx)$
with some uniqe $x \in S$ and
$$ \Hyp_-(Jx)^\circ = \cone\{x\}= \{tx:\,t\geq 0\} =[0,x ,$$
that is the polar of a hypercone is a closed halfline issuing from $0$.

We have similarly
$$\Hyp_+(J(x))^\circ =[0,-x .$$

Since 
$$\Hyp(J(x))=\Hyp_-(J(x)) \cap \Hyp_+(J(x)),$$
it follows that
$
\Hyp(J(x))^\circ= \Hyp_-{J(x)}^\circ  \cup \Hyp_+{J(x)}^\circ =[0,x \cup [0,-x =\spa \{x\}.
$

This means that the polar of the hyperplane
is a stright line= a one dimensional vector space.
\end{proposition}

\begin{proof}

Since $\|a\|^*=1$, from the reflexivity and strict convexity of $X$,  by James' characterisation
of strictly convex Banach spaces  there exists a unique $x \in S$
with 
$$\|a\|^*=\lng a,x \rng
=\|a\|^*\|x\| $$
hence, by the definition of $J$, we must have $a=Jx$.

Take $ y\in \Hyp_-(J(x))$, that is $\lng Jx,y\rng \leq 0.$
Hence we have
$$\|x-0\|=\lng Jx,x \rng \leq  \lng Jx,x-y \rng \leq \|Jx\|^*\|x-y\|= \|x-y\|.$$

This shows that $x\in \Hyp_-(Jx)^\circ$ and by the homogeneitty of $J$ that
$$ \{tx: \,t\geq 0 \}= \cone \{x\}= [0,x  \subset \Hyp_-(Jx)^\circ.$$
From the smoothness of  $\Hyp_-(Jx)$, $S+x$ and its homothetic images with center $0$ are the only 
spheres tangent to $\Hyp_-(Jx).$
Hence $\Hyp_-(Jx)^\circ =[0,x $.

A similae reasoning yields that $\Hyp_+(J(x))^\circ =[0,-x $ and completes the proof.
\end{proof}

\begin{theorem}\label{projhyp}

The metric projection on the hyperplane of the uniformly smooth
 and uniformly convex Banach space is linear.

\end{theorem}

\begin{proof} 
The proof is the immediate consequence of Proposition \ref{felt} and  of item(ii) of Lemma  \ref{kovlin}.
\end{proof}


\section{The polar of the convex cone}

\begin{proposition}\label{metsz}

Each closed convex cone  $\Kup$ is the intersection of hypercones. If

$$ \Pi _\Kup := \{x\in S: \Kup \subset \Hyp_-(Jx)\},$$
then
$$\Kup^\circ =\cup_{x\in  \Pi_\Kup} [0,x = \cone \Pi_\Kup.$$
 
\end{proposition}

\begin{proof}
The first assertion is the particular case of the
assertion from Banach space geometry that in  the uniformly convex
and uniformly smooth Banach space every nonempty closed convex set is
the intersection  of closed halfspaces .

Since $\Kup \subset \Hyp_-(Jx),\,\,\forall \,\, x\in  \Pi_\Kup$,
from the fact that $^\circ$ is obviously antitonic set-mapping, it follows that
$$ \Kup^\circ \supset  \cup_{x\in  \Pi_\Kup} [0,x.$$ 
Suppose that $y\in \Kup^\circ$. We can suppose $\|y\|=1$. Then the hyperplane with the normal $Jy$ supports
$\Kup$ at $0$, and hence $\Kup \subset \Hyp_-(Jy)$.
This shows that $y\in \Pi_\Kup $ and completes the proof.

\end{proof}


\section{The proof of Theorem \ref{eloz}}

We follow in the proof the scheme : (i) $\Rightarrow $ (ii) $\Rightarrow $ (iii) $\Rightarrow $ (iv) $\Rightarrow $  (v) $\Rightarrow$ (i)
and (i) $\Leftrightarrow$ (vi).

\begin{proof}

The implication (i) $\Rightarrow $ (ii) is obvious.

(ii) $\Rightarrow $ (iii).

Take the bidimensional subspace  $\mc D = \spa \{a,\,b \}\subset  X^*$ with $a,\,b \in S^*$ linearly independent elements.

From (ii) the metric projection $P_\mc V $ onto the vector space $\mc V = \ker \{a,\,b\}$ is linear.
Since $P_\mc V$ and $I-P_\mc V$ are mutually polar by Theorem 3, using Propsition \ref{kovlin}
we conclude that $\mc V^\circ =(I-P_\mc V)X$ is a bidimensional vector space
(it is in fact  the  subspace generated by the set $\{J^{-1}a,J^{-1}b\}.$).

Observe that the set of normals to $\mc V$ is the set  $\mc C = \mc D\cap S^*$. Hence with

$$ \Pi _\mc V := \{J^{*-1}(c)\in S: \mc V \subset \Hyp_-(c)\},\,\, c\in \mc C$$
we have
$$\mc V^\circ =\cup_{J^{*-1}(c)\in  \Pi_\mc V} [0,.J^{*-1}(c) = (I-P_\mc V)X$$
by Proposition \ref{metsz}. This means that  

$$ \{J^{*-1}(c) :\,c\in \mc C\} \subset \mc V^\circ,$$
and since  $\mc V^\circ$ is a vector space and $J^{*-1}$ is homogeneous 
that is 
$$\spa J^{*-1}(\mc C)=J^{*-1}(\mc D)  \subset \mc V^\circ,$$
that shws that (iii) holds.

The implication (iii)$\Rightarrow$(iv)  is  Corollary 3 in \cite{Nemeth2025}.

(iv) $\Rightarrow$ (v) is the content of Proposition 4 in \cite{Nemeth2025}.

(v) $\Rightarrow$ (i),
The closed vector space $\mc V \subset X$ is a closed convex cone. Since $\ker P_\mc V = (I-P_\mc V)X=
 \mc V^\circ$ is convex and $P_\mc V$ is homogeneus it follows  that  $\mc V^\circ$ is a vector space
 and by Proposition \ref{kovlin}  we have that $P_\mc V$ is linear.
 
 Finally, the equivalence (i) $\Leftrightarrow$  (v) is the classical result in \cite{Singer}.

\end{proof}

\begin{remark}

The theoren contains the following implicite results:

\begin{enumerate}[(i)]

\item In the uniformly convex and uniformly smooth Banach space of dimension $\geq 3$
the metric projection on subspaces  are linear
if and only if the polars of convex cones are convex.
Each one of these conditions characterizes inner product spaces.

\item  In the uniformly smooth and uniformly convex Banach space of dimension $\geq 3$  the
normalized duality mapping maps bidimensional
subspaces in bidimensional ones if and only if it is
linear. This property characterizes inner product spaces,

\end{enumerate}

\end{remark}

\end{document}